\documentclass[12pt, reqno]{amsart}
\usepackage{amsmath, amsthm, amscd, amsfonts, amssymb, graphicx, color}
\usepackage[bookmarksnumbered, colorlinks, plainpages]{hyperref}
\hypersetup{colorlinks=true,linkcolor=red, anchorcolor=green, citecolor=cyan, urlcolor=red, filecolor=magenta, pdftoolbar=true}

\usepackage{mathrsfs}
\newtheorem{theorem}{Theorem}[section]
 \newtheorem*{theo}{Theorem}

\newtheorem{proposition}[theorem]{Proposition}

\theoremstyle{definition}
\newtheorem{definition}[theorem]{Definition}
\newtheorem{example}[theorem]{Example}

\theoremstyle{remark}

\numberwithin{equation}{section}

 \def\bsr{\operatorname{bsr}}
\def\tsr{\operatorname{tsr}}
\def\ind{\operatorname{ind}}
 \def\bs{\boldsymbol}

 \def\T{ \mathbb T}
 \def\R{ \mathbb R}
 \def\H{H^\infty}

 \def\D{{ \mathbb D}}
  \def\K{{ \mathbb K}}
 \def\C{{ \mathbb C}}

 \def\N{{ \mathbb N}}

  \def\F{{ \mathscr F}}
 
 \def\e{\varepsilon}

 \def\bs{\boldsymbol}
 \def\dis{\displaystyle}
 \def\union{\cup}
 \def\Union{\bigcup}
 \def\inter{\cap}
 \def\Inter{\bigcap }
 \def\ov{\overline}

 \def\ss{\subseteq}
 \def\emp{\emptyset}

 \def\buildrel#1_#2^#3{\mathrel{\mathop{\kern 0pt#1}\limits_{#2}^{#3}}}
 
 \def\ssi{\Longleftrightarrow}

\begin{document}
\setcounter{page}{1}

\title[]{The cone and cylinder algebra}

\author[R. Mortini, R. Rupp]{Raymond Mortini$^1$$^{*}$ and Rudolf Rupp$^2$}

\address{$^{1}$ Universit\'{e} de Lorraine,
 D\'{e}partement de Math\'{e}matiques et  
Institut \'Elie Cartan de Lorraine,  UMR 7502,
 Ile du Saulcy,
 F-57045 Metz, France}
\email{\textcolor[rgb]{0.00,0.00,0.84}{raymond.mortini@univ-lorraine.fr}}

\address{$^{2}$  Fakult\"at f\"ur Angewandte Mathematik, Physik  und Allgemeinwissenschaften,
 TH-N\"urnberg,
 Kesslerplatz 12, D-90489 N\"urnberg, Germany}
\email{\textcolor[rgb]{0.00,0.00,0.84}{Rudolf.Rupp@th-nuernberg.de}}



\date{Received: xxxxxx; Revised: yyyyyy; Accepted: zzzzzz.
\newline \indent $^{*}$ Corresponding author}

\begin{abstract}
In this exposition-type note we present detailed proofs of certain assertions concerning several algebraic properties of the cone and cylinder algebras. These include   a determination
of the maximal ideals, the solution of the B\'ezout equation and a computation of the
stable ranks by elementary methods. 

  \end{abstract}

 \maketitle

 Let $\D=\{z\in \C: |z|<1\}$ be the  open unit disk in the complex plane $\C$ and ${\bf D}$  its closure. 
 As usual, $C(\bf D,\C)$ denotes  the space of continuous, complex-valued functions on $\bf D$ and
 $A(\bf D)$ the disk-algebra, that is the algebra of all functions in $C(\bf D,\C)$ which are holomorphic in $\D$.  By the Stone-Weierstrass Theorem, 
 $C({\bf D},\C)=[\,z,\ov z\,]_{{\rm alg}}$
and $A({\bf D})=[\,z]_{{\rm alg}}$, the uniformly closed subalgebras generated by $z$ and $\ov z$,
respectively $z$ on $\bf D$.
In this expositional note we study the  uniformly closed subalgebra
$$A_{co}=[\,z,|z|\,]_{{\rm alg}}\ss C(\bf D,\C)$$ 
 of  $C(\bf D,\C)$ which is generated by $z$ and $|z|$ as well as the algebra
  $$\mbox{${\rm Cyl}({\D})=\big\{f\in C\big({\bf D}\times [0,1],\C\big): \;f(\cdot, t)\in A(\bf D)\;$ for all
$t\in [0,1] \big\}$}.$$
We will dub the  algebra $A_{co}$ the {\it cone algebra} and the algebra ${\rm Cyl}({\D})$
the {\it cylinder algebra}.

 The reason for chosing these names will become clear  later in Theorem \ref{theconealg} and Proposition \ref{cyl}.   The cone-algebra appeared first in \cite{roy}; the cylinder algebra
 was around already at the beginning of the development of Gelfand's theory.
 Their role was mostly reduced to the category ``Examples" to  illustrate the general Gelfand theory;
 no detailed proofs appeared though. We think that these algebras deserve  a thorough analysis
 of their interesting algebraic properties and that is the aim of this note.  The main new results
 will be the determination of the stable ranks for the cone-algebra, the absence of  the corona property 
  $(Cn)$ when (and only when) $n\geq 2$,   explicit examples of peak functions and a  
  simple approach to the calculation of the Bass stable rank for the cylinder algebra.
 
 This paper  forms part of an ongoing textbook project of the authors
 on stable ranks of function algebras,
due to be finished only in a couple of years from now (now = 2015).  Therefore we decided
to  make this chapter already  available to the mathematical community (mainly for readers
of this special issue of Annals of Functional Analysis (AFA) dedicated to Professor Anthony To-Ming Lau and for master students interested  in function theory and function algebras).

 \section{The cone algebra}

We begin with some examples of non-trivial elements in $A_{co}$.  
The general situation will be dealt with
in Theorem \ref{repa5}.

\begin{example}\label{exacone}\hfill

\begin{enumerate}
\item For every $\delta>0$, $q>0$  and $p\in\R$, $\bigl(|z|^q+\delta\bigr)^p\in (A_{co})^{-1}$.
\item $f(z):=\begin{cases} \dis \frac{z}{\sqrt{|z|}} & \text{if $0<|z|\leq 1$}\\ 0 & \text{if $z=0$}
\end{cases}$
belongs to $A_{co}$. 
\end{enumerate}
\end{example}
\begin{proof}
(1) Since by Weierstrass' Theorem  $(x^q+\delta)^{\pm p}$ is a uniform limit of polynomials
 in $x=|z|\in [0,1]$, we have  that $(|z|^q+\delta)^p\in (A_{co})^{-1}$.
 
(2) First we note that $f$ is continuous on ${\bf D}$. For $\delta>0$,  let 
$$f_\delta(z):=\frac{z}{\sqrt{|z|}+\delta}.$$
Then, by (1), $f_\delta\in A_{co}$. Since $||f-f_\delta||_\infty\leq \delta/(1+\delta)$ (see below),
we conclude that $f$ is a uniform limit of functions in $A_{co}$; hence $f$ itself belongs to $A_{co}$.
Now we prove our estimate:
$$\left| \frac{z}{\sqrt{|z|}}- \frac{z}{\sqrt{|z|}+\delta}\right|= |z|\; \frac{\delta}{\sqrt{|z|}\,(\sqrt{|z|}+\delta)}
= \frac{\sqrt{|z|}}{\sqrt{|z|}+\delta}\;\delta=:h\left(\sqrt{|z|}\right).
$$
Since the derivative of $h(x)= \frac{\delta x}{x+\delta}$, namely  $h'(x)=\frac{\delta^2}{(x+\delta)^2}>0$ on $[0,1]$, we have
$h(x)\leq h(1)= \delta/(\delta+1)$.
\end{proof}

Next, we derive several Banach algebraic properties for $A_{co}$.

\begin{definition}\label{invclo}
Let $X$ be a topological space and $A$ a subalgebra of $C(X,\C)$.  
\begin{enumerate}
\item [(1)] The set of invertible $n$-tuples in $A$ is denoted by $U_n(A)$; that is
$$U_n(A)=\{(f_1,\dots,f_n)\in A^n: \exists (r_1,\dots,r_n)\in A^n: \sum_{j=1}^n r_jf_j=1\}.$$
 \item[(2)] $A$ is said to be {\it inverse-closed} (on $X$), if
$f\in A$ and $|f|\geq \delta>0$ on $X$  imply that $f$  is invertible in $A$.
\item [(3)] 
$A$ is said to  satisfy   condition (Cn) \footnote{ Here  (Cn) stands for ``Corona-condition
 for $n$-tuples''.} if
$$U_n(A) \dis =\Bigl\{(f_1,\dots,f_n)\in A^n: \Inter_{j=1}^n Z_X(f_j)=\emp\Bigr\},$$
where $Z_X(f)=\{x\in X: f(x)=0\}$ denotes the zero set of $f$.
\item [(4)] If $A$ is a commutative unital Banach algebra over $\C$, then its spectrum 
(or set of multiplicative linear functionals on $A$ endowed with the weak-$*$-topology) is denoted by $M(A)$. Moreover,  $\widehat A$  is the set of Gelfand transforms $\widehat f$ of elements 
in $A$.
\end{enumerate}
\end{definition}

As usual,  for a compact set $X$ in $\C$, $P(X)$ is the uniform closure in $C(X,\C)$
of the set $\C[z]$ of polynomials.

\begin{theorem}\label{theconealg}
Let $A_{co}=[\,z,|z|\,]_{{\rm alg}}\ss C(\bf D,\C)$ be the cone algebra.
Then 
\begin{enumerate}
\item  [(1)] $A({\bf D})\ss A_{co}\ss C(\bf D,\C)$.
\item [(2)] $A_{co}|_{\T}=A({\bf D})|_{\T}=P(\T)$.
  \item [(3)] For every $0<r< 1$, $A_{co}|_{r\T}=P(r\T)$. 
\item [(4)] $M(A_{co})$ is homeomorphic to the cone 
$$K:=\{(x,y,t)\in \R^3: \sqrt{x^2+y^2}\leq t,\; 0\leq t\leq 1\},$$
a three dimensional set.\\

 \begin{figure}[h]
   \hspace{3cm}
   \scalebox{.40} {\includegraphics{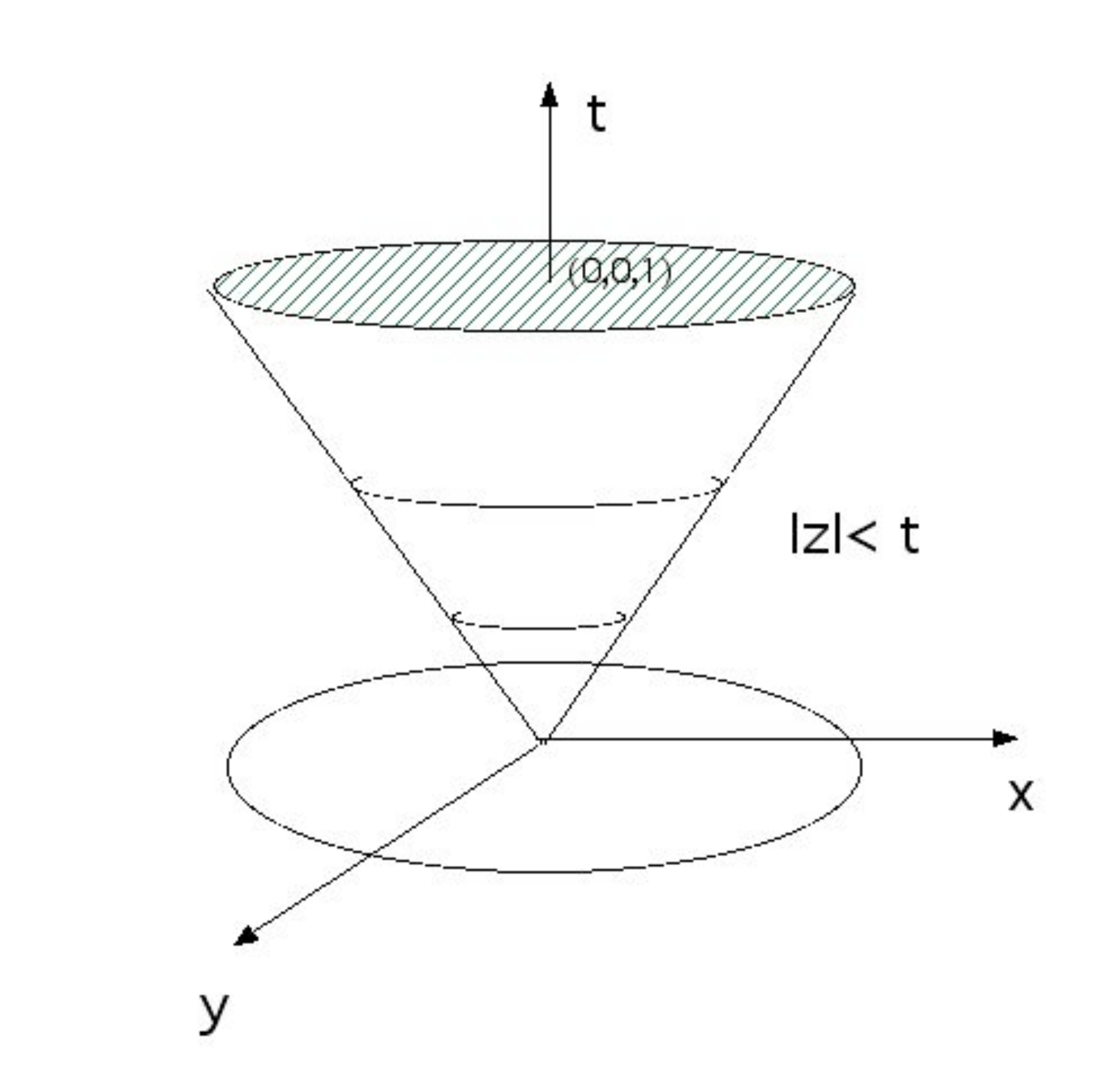}} 
\caption{\label{spectralcone} The cone as spectrum}
\end{figure}

\item[(5)]  For $f\in C({\bf D},\C)$ and $0<r<1$,  let $f_r$ be the dilation of $f$ given by
$f_r(z)=f(rz)$.  Moreover, let 
\begin{equation}\label{poiR}
P_R[f](z):=\frac{1}{2\pi} \int_0^{2\pi} \frac{R^2-|z|^2}{|Re^{it}-z|^2}\, f(Re^{it})\;dt
\end{equation}
be the Poisson integral and  $P[f]:=P_1[f]$. Then
$$M(A_{co})=\delta_0 \union \bigl\{\psi_{r,a}: 0< r\leq 1,\; a\in {\bf D}\; {\rm with}\; |a|\leq r\bigr\},$$
 where $\delta_0$ is point evaluation at $0$,  $\psi_{r,a}=\delta_a$ if $|a|=r$ and
$$\psi_{r,a}: \begin{cases} A_{co} &\to \C\\ f &\mapsto P[(f_r)|_\T](a/r)=
P_r[f|_{r\T}](a),\end{cases}$$
if $|a|<r$.  

\item [(6)] $\widehat A_{co}$ is the uniform closure of the polynomials $p(z,r)$ on the cone 
$$K=\{(z,t)\in \mathbf D\times [0,1]: |z|\leq t\}$$ and coincides with 
$$\mathcal A:=\bigl\{f\in C(K,\C): f(\cdot,r)\in A(r{\bf D})\;\; \forall\;  r\in \;]0,1]\;\bigr\}.$$

\item [(7)] $A_{co}$ is  inverse-closed; that is, it has property (C1), but    $A_{co}$ does not have property 
(Cn) for any $n\geq 2$.

\item[(8)] The Shilov boundary of $A_{co}$ coincides with the outer surface 
$$S:=\{(z,r)\in \C\times \R:  0\leq r\leq 1, |z|=r\}$$
of the cone $K$ (this is the boundary of $K$ without
the upper disk $\{(z,1)\in \C\times \R: |z|<1\}$).
The Bear-Shilov boundary  is the closed unit disk \footnote{ Recall that if
$X$ is a compact Hausdorff space and $L$ a point-separating $\K$-linear subspace of $C(X,\K)$
with $\K\ss L$,  then $L$ admits a smallest closed boundary, which we will call the Bear-Shilov boundary (see \cite{bea}). 
The Shilov boundary of a uniform algebra  $A$ is the smallest closed boundary
of $A$ on its spectrum $M(A)$.}.
\end{enumerate}
\end{theorem}

 \begin{proof}
(1)  This is clear  since the polynomials in $z$ are dense in $A(\bf D)$.

\centerline{\rule{6cm}{0,2mm}}\medskip

(2) Let $f\in A_{co}$. If we choose a sequence of polynomials $p_n\in \C[z,w]$ such that
$p_n(z,|z|)$ converges uniformly to $f$ on $\bf D$,  then $p_n(z,1)$  converges
uniformly on $\T$ to $f|_\T$. Hence, $f|_\T\in P(\T)=A(\bf D)|_{\T}$.
Together with (1), we conclude that $A({\bf D})|_{\T}= A_{co}|_{\T}$. 

\centerline{\rule{6cm}{0,2mm}}\medskip

(3) Fix $0<r<1$ and let $f\in A_{co}$.  Then, for $|z|=r$,  $f(z)=\lim p_n(z,r)\in P(r\T)$.
Hence $A_{co}|_{r\T}\ss P(r\T)$.
Conversely, given $h\in P(r\T)$, we let $H:=P_r[h]$ be the Poisson extension of $h$ to $r \D$.
Then $H\in A(r\D)$.  Now we define the function $f$ by
$$f(z)= \begin{cases}  H(z) &\text{ if $|z|\leq r$}\\ \dis H\Bigl(r\frac{z}{|z|}\Bigr)& \text{ if $r\leq |z|\leq 1$}.
\end{cases}
$$
Note that $f$ is an extension of $H$ to the unit disk that stays constant on every ray 
$se^{i\theta}$ beginning at the radius $r$.
Then $f$ is continuous on ${\bf D}$. Now $f(z)$ can be written as $f(z)=H\bigl(z g(|z|)\bigr)$,
 where $g$ is defined by 
$$g(s)=\begin{cases} 1 &\text{ if $0\leq s\leq r$}\\ \dis \frac {r}{s} & \text{ if $r\leq s\leq 1$}.
\end{cases}
$$
Then $g$ is continuous on $[0,1]$.  Next, we uniformly approximate on $[0,1]$ the function $g$ by a sequence  $(q_n)$ of polynomials in $\C[s]$  and $H$ on $\{|z|\leq r\}$ by a sequence of polynomials $(p_n)\in \C[z]$. Let
$$Q_n(s):=\frac{rq_n(s)}{\dis\max_{0\leq s\leq 1} |sq_n(s)|}$$ 
Then also $(Q_n)$ converges  uniformly to $g$ on $[0,1]$, because $\max_{0\leq s\leq 1}|sg(s)|=r$.
What we have gained is that $|z\, Q_n(|z|)|\leq r$ for every $z\in {\bf D}$.  
Hence $H(z\,Q_n(|z|))$ is well defined on ${\bf D}$ and $H(z\,Q_n(|z|))$ converges uniformly 
on ${\bf D}$ to $f(z)$.
 We claim that $p_n(z Q_n(|z|))$ converges uniformly on ${\bf D}$ to $f(z)$, too. In fact,
\begin{eqnarray*}|p_n(zQ_n(|z|)) -f(z)| &\leq& |p_n( zQ_n(|z|))-H( zQ_n(|z|))|+|H(zQ_n(|z|))-f(z)|\\
&\leq & \max_{|w|\leq r}|p_n(w)-H(w)| + \max_{|\xi|\leq 1}|H(\xi Q_n(|\xi|))-f(\xi)|.
\end{eqnarray*}
Now $Q_n(|z|)\in A_{co}$ implies $zQ_n(|z|)\in A_{co}$ and so $p_n(zQ_n(|z|)\in A_{co}$.
We conclude that $f\in A_{co}$. Since $f|_{r\T}=H|_{r\T}=h$, we are done: $P(r\T)\ss A_{co}|_{r\T}$.

\centerline{\rule{6cm}{0,2mm}}\medskip

(4)  Here we show that the spectrum of $M(A_{co})$
  is homeomorphic to the cone  (see figure \ref{spectralcone})
  $$K:=\{(x,y,t)\in \R^3: \sqrt{x^2+y^2}\leq t, 0\leq t\leq 1\},$$
   To this end,  we first note  that with $B:=C({\bf D},\C)$, 
$$\sigma_B(z,|z|)=\{(a_1,a_2)\in\C^2: |a_1|\leq 1, a_2=|a_1|\},$$
because $$(z-a, |z|-b)\in U_2(C({\bf D},\C))$$
 if and only if the functions $z-a$ and $|z|-b$ have no common zeros in $\bf D$.
Geometrically speaking,  $S:=\sigma_B(z,|z|)$ 
is the  surface of the cone in figure \ref{spectralcone}, without the upper basis
 ${\{(w,1)\in \C\times \R: |w|<1\}}$; we call this the {\it outer surface} of $K$. 
By a general Theorem in Banach algebras (see \cite{rud}), 
 $\sigma_{A_{co}}(z,|z|)$  now is the polynomial convex hull $\widehat S$ of
$S=\sigma_B(z,|z|)$, which we are going to determine below. 
Observe that  $K$ can be identified  with  the following compact subset of $\C^2$: 
$$\tilde C:=\{(z_1,z_2)\in \C^2: |z_1|\leq {\rm Re}\, z_2, \; 0\leq {\rm Re }\, z_2\leq 1,\;
{\rm  Im}\;z_2=0\},$$ 
and that $S\ss \tilde C\ss\R^3\times \{0\}$, because 
$$S=\{(z_1,z_2)\in \C^2: |z_1|=z_2={\rm Re}\, z_2,\; 0 \leq {\rm Re}\, z_2 \leq 1,\;  {\rm Im}\, z_2 =0\}.$$

Fix $0<t\leq 1$. We first show that every disk $D_t:=\{(w,t)\in \C^2: |w|\leq t\}$ 
is contained in $\widehat S$. To this end, fix $(w,t)\in D_t$ and 
consider any polynomial $p\in \C[z_1,z_2]$. Then
\begin{eqnarray*}
|p(w,t)| &\leq &\max \{|p(z_1,t)|:|z_1|\leq t\}=\max \{|p(z_1,t)|: |z_1|=t\}\\
&\leq &\max\{|p(z_1,z_2)|: (z_1,z_2)\in S\}.
\end{eqnarray*}
Hence $(w,t)\in \widehat S$ and so $D_t\ss \widehat S$. Consequently,
$$S\ss \tilde C=\Union_{|t|\leq 1} D_t \ss\widehat S.$$

 Since $K$ is a convex set in $\R^3$, $\tilde C$ is a
convex compact set in $\C^2$, and so $\tilde C$ is polynomially convex.
Thus $\widehat S=\tilde C$. We conclude that
$(z-a,|z|-b)$ does not belong to $U_2(A_{co})$ if and only if  $b\in [0,1]$ and $|a|\leq b$.\medskip

\centerline{\rule{6cm}{0,2mm}}\medskip

(5)   Let $m\in M(A_{co})$ and put $a:=m(z)$ and $r:=m(|z|)$. Then $(a,r)\in \sigma_{A_{co}}(z,|z|)$.
Hence, by (4), $r\in [0,1]$,  $a\in{\bf D}$ and $|a|\leq r$. That is, $|m(z)|\leq m(|z|)$.
Let $f\in A_{co}$ and $p_n\in\C[z,w]$ a sequence
of polynomials such that $p_n(z,|z|)$ converges uniformly on $\bf D$ to $f(z)$. 
By (3), $f|_{r\T}\in P(r\T)$.  Now 
\begin{equation}\label{randkegel}
\mbox{$\lim p_n(z,r)=f(z)$ (uniformly on $|z|=r$)}.
\end{equation}
By the maximum principle, $p_n(\xi,r)$ converges  uniformly on $r\D=\{|\xi|\leq r\}$ to a function
$\check f$ with $\check f=f|_{r\T}$. Moreover, 
$$
\mbox{$\check f(w)=P[(f_r)|_{\T}](w/r)= P_r[f|_{r\T}](w)$ for $|w|< r$}.
$$
On the other hand, because $m(z)=a$ and $m(|z|)=r$,
$$m(f)=\lim m(p_n)= \lim p_n(a,r).$$
Hence, if $|a|< r$, we conclude that $m(f)=\check f(a)$. In other words, $m=\psi_{r,a}$.
If $|a|=r$, then this limit $m(f)$ coincides with $f(a)$ by \eqref{randkegel}; that is $m=\delta_a$.

\begin{figure}[h]
   \hspace{0cm}
   \scalebox{.40} {\includegraphics{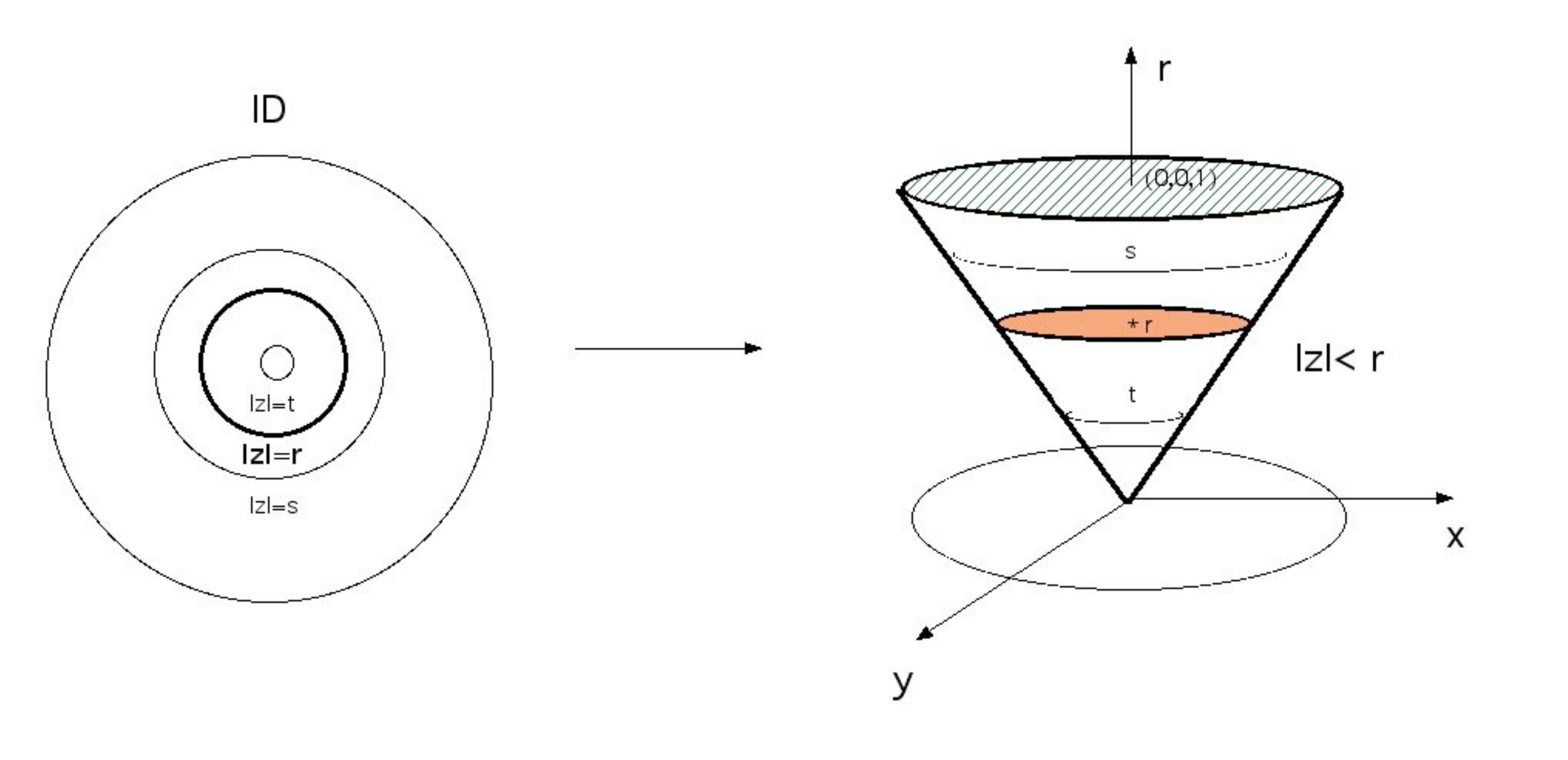}} 
\caption{\label{spectr} Functionals on the disk correspond to functionals 
on the  surface of the cone }
\end{figure}

It remains to show  the converse, that is  $\psi_{r,a}=\delta_a\in M(A_{co})$ when $|a|=r$ (which is clear) and 
$\psi_{r,a}\in M(A_{co})$ for every $(a,r)\in K$ 
with $|a|<r$ and $0<r\leq 1$. To this end, 
let $T_r: A_{co}\to A(\D)$ be the map given by $T_r(f)=P[(f_r)|_{\T}]$. Since $(f_r)|_\T\in P(\T)$,
and since the Poisson operator is multiplicative on $P(\T)$,
we deduce that $T_r$ is an algebra homomorphism. Hence, for every $\xi\in\D$, the
map  $m$ given by  
$m:=\delta_\xi\circ T_r$
is a homomorphism of $A_{co}$ into $\C$. Since $m({\bf 1})=1$, we conclude that $m\in M(A_{co})$.
If $\xi\in \D$ is chosen so that $a=r\xi$, then $m= \psi_{r,a}$.\medskip

\centerline{\rule{6cm}{0,2mm}}\medskip

(6)\;  
>From (5) we conclude that the Gelfand transform
$\widehat f$ of $f$ has the following properties: ${\widehat f(0+i\,0,0)=f(0)}$ and,  if $0<r\leq 1$,
then  $\widehat f (w,r)=f(w)$ whenever  $|w|=r$, $w\in\C$,  and  $\widehat f(w,r)=P_r[f|_{r\T}](w)$
whenever $|w|<r$. Since $f|_{r\T}\in P(r\T)$ is the boundary function of a holomorphic function
on $|w|<r$, we conclude from (3) that $\widehat f(\cdot,r)\in A(r\,{\bf D})$.   Hence
$$\widehat A_{co}\ss \mathcal A=\bigl\{f\in C(K,\C): f(\cdot,r)\in A(r{\bf D})\;\; \forall\;  r\in \;]0,1]\bigr\}.$$
Next we observe that
  $\widehat z=z$ and 
$\widehat{|z|}=t$ on $K=\{(z,t)\in {\bf D}\times [0,1], \; |z|\leq t\}$. Hence, if $p\in \C[z,w]$,
then with $P(z):=p(z,|z|)$, $z\in {\bf D}$,  we see that
$$\widehat P(z,r)=p(z,r).$$
Since $A_{co}$ is a uniform algebra,  $(\widehat A_{co},||\cdot||_{M(A)})$  is isomorphic isometric
to $A_{co}$. Thus $\widehat A_{co}$ is the closure $\mathcal P(K)$ of the polynomials of the form $p(z,r)$ within $C(K,\C)$.
That  $\mathcal P(K)=\mathcal A$ ~~ follows immediately from Bishop's antisymmetric decomposition theorem    \cite[p. 60]{gam} and the fact
that the maximal antisymmetric sets for  $\mathcal P(K)$   as well as $\mathcal A$
\footnote{ Recall that a closed subset 
$E$ of $K$ is said to be a set of antisymmetry for a function algebra $A\ss C(K,\C)$ 
if every function in $A$ which is real-valued on $E$, is already constant on $E$.} 
 are the disks 
$$D_t=\{(w,t)\in \mathbf D: |w|\leq t\},$$
with $0<t\leq 1$ and the singleton $\{(0+i0,0)\}$.
An elementary proof can  be given   along the  same lines  as in Theorem \ref{repa5}.\medskip

\centerline{\rule{6cm}{0,2mm}}\medskip

(7) Suppose that  $f\in A_{co}$ has no zeros on $\bf D$. By a theorem of Borsuk 
 (see \cite[p.~99]{bu})
 $f$ has a continuous logarithm 
on ${\bf D}$, say $f=e^g$ for some $g\in C({\bf D},\C)$. Let
$$N:=\ind (f_r)|_{\T}= n(\widehat f(\cdot,r),0)$$
be the index (or winding number) of $h:=(f_r)|_{\T}$ (see \cite[p.~84]{bu}).
  Note that $h$ has no zeros on $\T$. Hence, $N$ is well defined. 
This number, though,  coincides with the number of zeros of the holomorphic 
function $\widehat f(\cdot,r)$
in $r\D$.  But 
$$\ind (f|_{r\T})=\ind (e^g|_{r\T})=\ind (e^{g_{_r}}|_\T)=0.$$
Thus the Gelfand transform $\widehat f$ of $f$ does not vanish on $M(A)$.   Hence $f$ is invertible
in $A_{co}$. In other words, property (C1) is satisfied. 

Next we show that (C2) is not satisfied. In fact, consider the pair $(z, 1-|z|^2)$. Although this
pair is invertible in $C({\bf D},\C)$, it is not invertible in $A_{co}$. To see this, we assume the contrary.
Thus, there  exist $a,b\in A_{co}$ such that
$$a(z)\, z+ b(z) (1-|z|^2)=1.$$
In particular, $a(z)\,z=1$ for $|z|=1$.  In other words, $a(z)=\ov z$. But by (2), ${A_{co}}|_\T=P(\T)$.
Since $\ov z\not\in P(\T)$ (otherwise $P(\T)$ would coincide with $C(\T,\C)$), we have obtained
a contradiction. Thus we have found a pair $(f,g)$ of functions in $A_{co}$ without common zeros on
 $\bf D$ but for which $(f,g)\notin U_2(A_{co})$.  This implies that $A_{co}$ does not have property $(Cn)$ for any  $n\geq 2$.\\
Here is another way to see that  $(z,1-|z|^2)$ is not in $U_2(A_{co})$. Let $\psi_{1, 0}\in M(A)$ 
be  the functional in (5), where $r=1$ and $a=0$. Then  $\psi_{1,0}(f)=P[f|_\T](0)$. But if $f(z)=1-|z|^2$,
then $f\equiv 0$ on $\T$ and so $\psi_{1,0}(1-|z|^2)=0$. But $\psi_{1,0}(z)=0$, too, since
$P[e^{it}](w)= w$ for every $w\in \D$. Thus $z$ and $1-|z|^2$ both belong to the kernel of a 
multiplicative linear functional on $A_{co}$. \medskip

\centerline{\rule{6cm}{0,2mm}}\medskip

(8) We first determine then Bear-Shilov boundary of $A_{co}$.
To this end, it is sufficient to 
show that every $a\in \bf D$ is a peak-point for $A_{co}$. So 
let $a\in\bf D$. If $a=0$, then we take the peak functions $f(z)= \frac{1}{1+|z|}$ or $f(z)=1-|z|$.
If $a\not=0$, then we first choose a peak function $p(x)\in \C([0,1],\R^+)$ with $p(x)\leq Cx$,
$p(|a|)=1$ and $0\leq p(x)<1$ for $x\in [0,1]\setminus \{|a|\}$. Now let
$$f(z):=\begin{cases} \left(1+ \dis\frac{z}{|z|}\,e^{-i\arg a}\;\right)\, \dis\frac{p(|z|)}{2} & \text{if $z\not=0$}\\ 0 & \text{if $z=0.$}
\end{cases}$$
Then $f\in C(\bf D,\C)$, $|f|\leq 1$, and if $|z|\not=|a|$,  then $|f(z)|\leq p(|z|)<1$. If $|z|=|a|$, then,
with $z=|a|e^{it}$, 
$$\mbox{$\left|1+ \dis\frac{z}{|z|}\,e^{-i\arg a}\;\right|= |1+e^{i(t-\arg a)}|<2$ if $t\not=\arg a \mod 2\pi$}.$$
Hence $|f(z)|<1$ for all $z\in {\bf D}\setminus \{a\}$.

It remains to show that $f\in A_{co}$. To this end, it suffices to prove that $zp(|z|)/|z|\in A_{co}$.
According to Weierstrass' theorem, let $P_n\in \C[x]$ be a sequence of  
 polynomials uniformly converging on $[0,1]$ to $p(x)$. Then $P_n(|z|)$ converges
 uniformly to $p(|z|)$ on $\bf D$. Since $P_n(|z|)\in A_{co}$, its limit $p(|z|)\in A_{co}$, too.
 Now $\delta+|z|\in (A_{co})^{-1}$   for every $\delta>0$  by Example \ref{exacone}.
 Hence $q_\delta(z):=zp(|z|)/(\delta +|z|)\in A_{co}$.
 But $\lim_{\delta\to 0}  q_\delta(z)=zp(|z|)/|z|$ uniformly on $\bf D$, because
 $$\left| \frac{z\, p(|z|)}{\delta+|z|}- \frac{z\,p(|z|)}{|z|}\right|= p(|z|)\frac{\delta}{\delta+|z|}\leq
\frac{C|z|}{\delta+|z|}\delta\leq C\, \delta\to 0.
$$
Thus we have shown that $f$ is a peak function at $a$ in $A_{co}$.  Here is a different example.
For $0<|a|\leq 1$, let
 $$f(z)=a+ z e^{-\frac{|z|}{|a|}}.$$
Then $f\in A_{co}$  and $f$ takes its maximum modulus in ${\bf D}$  only at $z=a$. In fact,
since the function $xe^{-x}$ takes its maximum on $[0,\infty[$ at $x=e^{-1}$,
\begin{eqnarray*}
|f(z)| &\leq& |a| + |z| e^{-\frac{|z|}{|a|}}=|a|\bigl( \mbox{$1+ \frac{|z|}{|a|}$}\,e^{-\frac{|z|}{|a|}}\bigr)
\leq |a|\, (1+e^{-1}),
\end{eqnarray*}
with equality at $z=a$.  Now  the last inequality is strict for $|z|\not=|a|$. 
If $z=|a|e^{it}$ and $a=|a|e^{i\arg a}$,  then
$$|f(z)|= |a|\, |e^{i\arg a}+e^{it}e^{-1}|<|a|( 1+e^{-1}) \ssi   \arg a \not= t \mod 2\pi.$$
Hence $f$ takes its maximum modulus only at $a$ and so $a$ is a peak point for $A_{co}$.
We conclude that the Bear-Shilov boundary is ${\bf D}$.  
Moving to $\widehat A_{co}$, we get that 
$$\widehat f(w,r)=\begin{cases} \left(1+ \dis\frac{w}{r}\,e^{-i\arg a}\;\right)\, \dis\frac{p(r)}{2} & \text{if $(w,r)\in K$, $r\not=0$}\\ 0 & \text{if $w=r=0.$}
\end{cases}$$
is a peak function at $(a,|a|)\in S\ss \partial K$. 
Since $\widehat g(\cdot,r)\in A(r{\bf D})$
for every $g\in A_{co}$, we just need to to apply the maximum  principle for 
holomorphic functions to conclude that no point
in the interior of the cone and on its upper surface $\{(w,1): |w|<1\}$ is a peak point for $\widehat A_{co}$.
Hence the Shilov boundary coincides with  the outer surface of the cone.
\end{proof}

Results on the peak sets for $A_{co}$ can be found in \cite{jim}.

 \begin{theorem}\label{repa5}
 We have the following identity:
 $$[\,z,|z|\,]_{{\rm alg}}=\{f\in C({\bf D},\C): f|_{r\T}\in P(r\T)\; \forall r\in \;]0,1]\}.$$
\end{theorem}
 \begin{proof}
 i) This follows from Bishop's antisymmetric decomposition theorem  \cite[p. 60]{gam} and the fact
that the maximal antisymmetric sets for $A_{co}$ are  the circles $\{|z|=r\}$, $0\leq r\leq 1$.
We would like  to present  the following   elementary proof, too.\\

 ii)  Let $A^*=\{f\in C({\bf D},\C): f|_{r\T}\in P(r\T)\}$.  We already know that 
 $[\,z,|z|\,]_{{\rm alg}}\ss A^*$. Observe that  every $h\in P(r\T)$
is the trace of a function $H$ that is continuous on $\{|z|\leq r\}$ and holomorphic in $\{|z|<r\}$. Hence
$H$ writes as $H(z)=\sum_{n=0}^\infty h_n z^n$, where $h_n$ are the Taylor coefficients of $H$.
They are given by the formula
\begin{eqnarray*}h_n &=&\frac{1}{n!} H^{(n)}(0) =\frac{1}{2\pi i}\int_{|\zeta|=r}
 \frac{H(\zeta)}{\zeta^{n+1}}\;d\zeta\\
 &=& \frac{1}{2\pi}\int_0^{2\pi}  \frac{h(re^{it})}{r^n} e^{-int}\; dt.
\end{eqnarray*}
The Fourier series  associated with $h$ then has the form 
$$h(re^{it})\sim \sum_{n=0}^\infty h_n r^ne^{int}.$$
Now fix $f\in A^*$. Since   $f|_{r\T}\in P(r\T)$, the preceding lines  imply that
for every $0<r\leq 1$ the Fourier series of the family of functions $f_r$
 \footnote{ Note that the Fourier series 
$\sum_{n=0}^\infty b_n(r) e^{int}$ for $(f_r)|_\T$ would not be useful to our problem here,
since at a later stage of the proof we really need the factor $r^n$.} are given by
$$f_r(e^{it})=f(re^{it})\sim \sum_{n=0}^\infty a_n(r)r^ne^{int}.$$
Here  
\begin{equation}\label{coeff-esti}
r^n|a_n(r)|\leq \bigl\Vert{(f_r)|}_\T\bigr\Vert_1\leq {||f||}_{\infty},
\end{equation}
where $||\cdot||_1$ is the $L_1$-norm on $\T$ and $||\cdot||_\infty$ the supremum norm on $\bf D$,
and
\begin{equation}\label{gegennull}
\mbox{$\lim_{r\to 0} r^na_n(r)=0$ for every $n=1,2, \dots$.}
\end{equation}
because
$$2\pi r^na_n(r)= \int_0^{2\pi} f(re^{it}) e^{-int}\, dt\buildrel\to_{r\to 0}^{} f(0)\int_0^{2\pi}e^{-int}\, dt =0.$$
Note also that  the map  $r\mapsto r^na_n(r)$, $n\in \N$,  is a continuous function on $[0,1]$; in fact,
since $f$ is uniformly continuous on $\bf D$,
\begin{eqnarray}\label{contiar}
2\pi |r^na_n(r)-r'^na_n(r')| &=& \left|\int_0^{2\pi} e^{-int} ( f(re^{it})-f(r'e^{it}))\, dt\right| \nonumber\\
&\leq & \int_0^{2\pi}| f(re^{it})-f(r'e^{it})|\, dt\leq 2\pi \e
\end{eqnarray}
if $|r-r'|<\delta.$
Consider now, for the parameter $r\in \;]0,1]$ and $0<\rho\leq 1$,  the polynomial  (in $z\in\C$)
$$p_N(z,r)=\sum_{n=0}^N a_n(r)  \rho^n z^n.$$
We show, that  
\begin{equation}\label{ppp}
\max_{|z|=r}|f(z)-p_N(z,r)|<\e
\end{equation} 
for suitably chosen  $\rho$, $\rho$ close to $1$, and
some $N\in \N$, $N$ and $\rho$ independent of $r$.

Let us start the proof of \eqref{ppp}. Since $f$ is uniformly continuous on $\bf D$, we may choose
 $\eta\in \;]0,1]$, independent of $r$,  so that $|f(re^{it})-f(re^{i\theta})|<\e$ for $|t-\theta|<\eta$ and every $r\in [0,1]$.
Fix $t\in [0,2\pi[$.  Let $I=I(t)\ss\T$ be  the arc centered at $t$ and with arc length $2\eta$.
Then {\footnotesize
\begin{eqnarray*}
\Bigl|f_r(e^{it})-P[(f_r)|_\T](\rho e^{it})\Bigr|&=&\frac{1}{2\pi}\left|\int_{0}^{2\pi} 
\frac{1-|\rho|^2}{|e^{i\theta}-\rho e^{it}|^2}\;(f_r(e^{it})-f_r(e^{i\theta}))\; d\theta\right| \\
&\leq& \frac{2||f||_\infty}{2\pi }\intop_{\{\theta: e^{i\theta}\in \T\setminus I\}}
\frac{1-|\rho|^2}{|e^{i\theta}-\rho e^{it}|^2}
\;d\theta
 +\frac{\e}{2\pi } \intop_{\{\theta: e^{i\theta}\in I\}} \frac{1-|\rho|^2}{|e^{i\theta}-\rho e^{it}|^2}\;d\theta.
 \end{eqnarray*}}
Note that the second integral is less than $\e$, because the integral $\int_0^{2\pi}P\,d\theta$
of the Poisson kernel is one.

Since for $\eta\leq |\theta-t|\leq \pi$ and $\rho$ close to $1$
\begin{eqnarray*}|e^{i\theta}-\rho e^{it}| &= &|e^{i(\theta-t)}-\rho| \geq |e^{i(\theta-t)}-1|-(1-\rho)\\
&=&2\, |\sin(\mbox{$\frac{\theta-t}{2}$})|-(1-\rho) \\
&\geq& \frac{2\eta}{\pi}-(1-\rho) \geq \frac{\eta}{\pi},
  \end{eqnarray*}
 we see that  the first integral tends to $0$ as $\rho\to 1$. 
Hence, for all $t$,
\begin{equation}\label{pp1}
\sup_{0<r\leq 1}\Bigl|f_r(e^{it})-P[(f_r)|_\T](\rho e^{it})\Bigr|<c\e
\end{equation}
for some $\rho$ sufficiently close to $1$ and independent of $r$.
Now
$$\check f(t):=P[(f_r)|_\T](\rho e^{it})=\sum_{n=0}^\infty a_n(r) r^n\rho^n e^{int}.$$
Because by \eqref{coeff-esti}
\begin{eqnarray*}
L:=\Bigl|\sum_{n= N+1}^\infty  a_n(r) r^n\rho^n e^{int}\Bigr|\leq ||f||_{\infty}\; \sum_{n=N+1}^\infty \rho^n,
\end{eqnarray*}
we see that $L<\e$ whenever $N$ is sufficiently large.  Note that $N$ is independent of $r$.
Hence
$$\left|\check f(t)-\sum_{n=0}^N a_n(r) r^n\rho^n e^{int}\right| <\e.$$
We conclude  from \eqref{pp1} that
\begin{equation}\label{pppeps}
\Bigl|f_r(e^{it})-\sum_{n=0}^N a_n(r) \rho^n\, r^ne^{int}\Bigr|<c\,\e+\e =\tilde c\, \e
\end{equation}
for every $r\in \,]0,1]$. This proves our claim \eqref{ppp}.

Now the coefficients $a_n(r) $ of the polynomial $p_N(z):=\sum_{n=0}^N a_n(r)\rho^nz^n$
are continuous functions for  $r\in \;]0,1]$ \footnote{ Note that, in general,  $a_n(r)$ is not continuous at $r=0$ for $n\geq 0$, as the function $f(z)= z/|\sqrt{|z|}=z/\sqrt r\in A_{co}$ shows.}
and $a_0(r)$ is continuous on $[0,1]$.
 In order to be able to use Weierstrass' approximation theorem,  we need to modify the $a_n(r)$ a little bit near the origin for $n\not=0$ by multiplying them with $r^N$, $r$ close to $0$.
 According to \eqref{gegennull}, for $\e>0$, there exists  $\delta>0$ such that for  $0\leq r\leq \delta$,
 $$\sum_{n=1}^N |a_n(r)r^n|<\e/2 .$$
 Let $\kappa\in C([0,1],[0,1])$ be defined as $\kappa(r)=r^N$ whenever $0\leq r\leq \delta/2$
 and $\kappa(r)=1$ for $\delta \leq r\leq 1$. Consider the functions
 $$p^*_N(z):=a_0(r)+\sum_{n=1}^N \kappa(r) a_n(r) \rho^nz^n.$$
Note that the new coefficients, $\kappa(r)a_n(r)$, are continuous on $[0,1]$, due to   \eqref{gegennull} and
\eqref{contiar}.
 Now, for   $\delta\leq r\leq 1$ and $|z|=r$,
\begin{eqnarray*}|p_N(z)-p^*_N(z)|  &\leq &|1-\kappa(r)|\;\Bigl|\sum_{n=1}^N  a_n(r)\rho^nz^n\Bigr|=0.
\end{eqnarray*}
If $0\leq r\leq \delta$ and $|z|=r$, then
\begin{eqnarray*}|p_N(z)-p^*_N(z)|  &=& |1-\kappa(r)| \Bigl|\sum_{n=1}^N(a_n(r) r^n)\rho^n e^{int}\Bigr|\\
&\leq& 2 \sum_{n=1}^N |a_n(r)|\,r^n \leq \e
\end{eqnarray*}
Hence $p_N^*$ is uniformly close to $p_N$.  Now 
for every $n\geq 1$,
there is a polynomial $q_n(r):=\sum_{j=0}^{M(n)}\; b_j(n) r^j$ such that
$$\max_{0\leq r\leq 1} |q_n(r)-\kappa(r) a_n(r)|<\e/(N+1).$$
For $n=0$ we choose $q_0\in \C[x]$ such that $|a_0(r)-q_0(r)|<\e/(N+1)$ on $[0,1]$.
Consequently, with $z=re^{it}$,
{\tiny\begin{eqnarray*}
\Bigl|f(re^{it})-\sum_{n=0}^N q_n(r)\rho^n r^ne^{int}\Big|&\leq &
\Bigl|f(re^{it})-\sum_{n=0}^N a_n(r)\rho^nr^ne^{int}\Bigr|+|p_N(z)-p_N^*(z)|
+\sum_{n=0}^N  \e/(N+1)\\
&\buildrel\leq_{}^{\eqref{pppeps}} & 
\tilde c\, \e+\e+\sum_{n=0}^N  \e/(N+1)=c^*\e.
\end{eqnarray*}}
In other words, for any $z\in {\bf D}$,
$$|f(z)-\sum_{n=0}^N q_n(|z|)\, \rho^nz^n|<c^*\e. $$
Thus $A^*=A_{co}$. We also deduce that
$$\widehat A_{co}=\bigl\{f\in C(K,\C): f(\cdot,r)\in A(r{\bf D})\;\; \forall\;  r\in \;]0,1]\;\bigr\}.$$
\end{proof}


\section{The stable ranks of the cone algebra}

The following concepts were originally introduced by H. Bass \cite{ba} and M. Rieffel \cite{ri}.
  \begin{definition}
  Let $A$ be a  commutative unital Banach algebra over $\R$ or $\C$.
\begin{enumerate}
\item [(1)] An $(n+1)$-tuple $(f_1,\dots,f_n,g)\in U_{n+1}(A)$ is  
called {\sl reducible}  (in $A$)  if there exists 
 $(a_1,\dots,a_n)\in A^n$ such that $(f_1+a_1g,\dots, f_n+a_ng)\in U_n(A)$.
\item [(2)] The {\sl Bass stable rank} \index{Bass stable rank} of $A$, denoted by $\bsr A$,  is the smallest integer $n$ such that every element in $U_{n+1}(A)$ is reducible. 
 If no such $n$ exists, then $\bsr A=\infty$. 
 \item [(3)] The {\it topological stable rank}, $\tsr A$, of $A$ is the least integer
  $n$ for which $U_n(A)$ is dense in $A^n$, or infinite if no such $n$ exists.  
  \end{enumerate}
  \end{definition}
We refer the reader to the work of L. Vasershtein \cite{va}, G. Corach and A. Larotonda 
\cite{cl,cl1}, G. Corach and D. Su\'arez \cite{cs, cs1,cs2,cs3}
and the authors 
dealing with numerous  aspects of these notions in the realm of function algebras. 
The computation of the stable rank of our algebras above  will be based on the following three results
from a higher analysis course:

\begin{theo}[A]\label{lipi}
Let $U\ss\R^n$ be open and  $f:U\to \R^n$ a map. 
Suppose that  $E\ss U$ has $n$-dimensional 
Lebesgue measure zero. Let $0<\alpha\leq 1$. Then,   under each of the following conditions,
$f(E)$ has $n$-dimensional Lebesgue measure zero, too:
 
\begin{enumerate}
\item [(1)] $f$ satisfies a H\"older-Lipschitz condition (of order $\alpha$) on $U$; that is, there is
$M>0$ such that 
 $$\mbox{$||f(x)-f(y)|| \leq M\, ||x-y||^\alpha$ for every $x,y\in U$}.$$
\item [(2)] $f\in C^1(U,\R^n)$.
\end{enumerate}
\end{theo}

A proof of the following version of Rouch\'e's theorem  (continuous-holomorphic pairs)
is in \cite[Theorem 20]{moru}.

\begin{theo}[B]\label{rch}
 
Let  $K\ss\C$ be compact, $f\in C(K,\C)$ and $g\in A(K)$, the set of all functions continuous on $K$
and holomorphic in the interior $K^\circ$ of $K$. Suppose that  on $ \partial K$
$$|f+g|<|f| +|g|.$$
Then  $f$ has a zero  on $K^\circ$ whenever   $g$ has a zero on $K^\circ$.
The converse does not hold, in general. 
\end{theo}

\begin{theo}[C]\label{nonextension}\cite[p. 97]{bu}.
Let $K\ss \C$  be compact, $C$ a  bounded component of $\C\setminus K$ and $\beta\in C$.
Then the function $f(z)=z-\beta$ defined on $K$ is zero-free on $K$, but does not
admit a zero-free extension to $K\union C$.
\end{theo}


\begin{theorem}
$\bsr A_{co}=2$ and $\tsr A_{co}=2$.
\end{theorem}
\begin{proof}
$\bullet$~~ We first show that $\tsr A_{co}\leq 2$. Let $(f,g)\in (A_{co})^2$. Choose polynomials $p(z,w)$ and $q(z,w)$\
in $\C[z,w]$ such that 
$$\mbox{$|f(z)-p(z,|z|)|+ |g(z)-q(z,|z|)|<\e $ ~~ for every $ z\in {\bf D}$.}$$
Let $P(z)=p(z,|z|)$ and $Q(z)=q(z,|z|)$. 
By the proof of assertion (6) of Theorem \ref{theconealg}, 
the Gelfand transforms of $P$ and $Q$  are polynomials, too, such that
$$\widehat P(z,r)=p(z,r) \;{\rm and}\; \widehat Q(z,r)=q(z,r).$$
We shall now use  Theorem (A). To this end, we observe that the function
$\widehat P$  and $\widehat Q$ satisfy a Lipschitz condition on  $K$.
Let $$\tilde K=\{(x,y, t,v)\in \R^4:\; v=0, \; 0\leq t\leq 1,\; \sqrt{x^2+y^2}\leq t \},$$
which is of course nothing else than our cone $K$, resp. $\tilde C$ (but viewed as a set in $\R^4$).
Then  $\tilde K$ has 4-dimensional Lebesgue measure 
zero. Now we look at the map
{\small$$\mu: \begin{cases} \tilde K \ss \R^4 &\to \R^4 \\ (x,y,t,v) & \mapsto 
( {\rm Re}\; p(x+iy,t), {\rm Im }\; p(x+iy,t), {\rm Re}\; q(x+iy,t), {\rm Im}\; q(x+iy,t)).\end{cases}
$$}
Then $\mu$ satisfies a Lipschitz condition on $\tilde K$, too. Hence, by  Theorem (A),
$\mu(\tilde K)$ has measure zero in $\R^4$.  
Thus there exists a nullsequence $(\e_n, \e'_n)$ in $\C^2$
such that  $p(z,t)-\e_n$ and $q(z,t)-\e'_n$ have no  common zero on 
$$K=\{(z,t)\in \mathbf D\times[0,1]: |z|\leq t\}.$$
Since $K=M(\widehat A_{co})\sim M(A_{co})$, these pairs are invertible in
 $\widehat A_{co}$ and so 
 $$\big(p(z,|z|)-\e_n,~q(z,|z|)-\e_n'\big )\in U_2(A_{co}).$$ 
 Hence $\tsr A_{co}\leq 2$.\\

$\bullet$~~ Next we  prove that $\tsr A_{co}\geq 2$.  Let $f(z)=z$. If we suppose that there exists
$u\in (A_{co})^{-1}$ such that $||u-z||_\infty<1/2$, then on $\T$
$$|u(z)-z|<\frac{1}{2}<1\leq |z|+|u(z)|.$$
Hence, by Rouch\'e's Theorem (B),  $u$ has a zero in ${\bf D}$. Thus, $u$ is not invertible in $A_{co}$. To sum up, we showed  that $\tsr A_{co}=2$.\\

$\bullet$~~ Since $A_{co}$ is a Banach algebra,  we have $1\leq \bsr A_{co}\leq \tsr A_{co}\leq 2$.
It remains to prove  that $\bsr A_{co}\geq 2$. 
The idea is to unveil a function $g\in A_{co}$ such that the zero set 
of $\widehat g$ on 
\begin{eqnarray*}
K&=& \{(z,t)\in \C\times [0,1]: |z|\leq t, 0\leq t\leq 1\}\\
&=&\{(x,y,t)\in \R^3: \sqrt{x^2+y^2}\leq t,\; 0\leq t\leq 1\}.
\end{eqnarray*}  
is a Jordan curve $J$ contained in the plane $y=0$ and a function $f\in A_{co}$
satisfying  $Z(\widehat f)\inter Z(\widehat g)=\emp$  such that $\widehat f$ is a translation 
of the identity map on $J$.

  \begin{figure}[h]
   \hspace{-1,3cm}
   \scalebox{.45} {\includegraphics{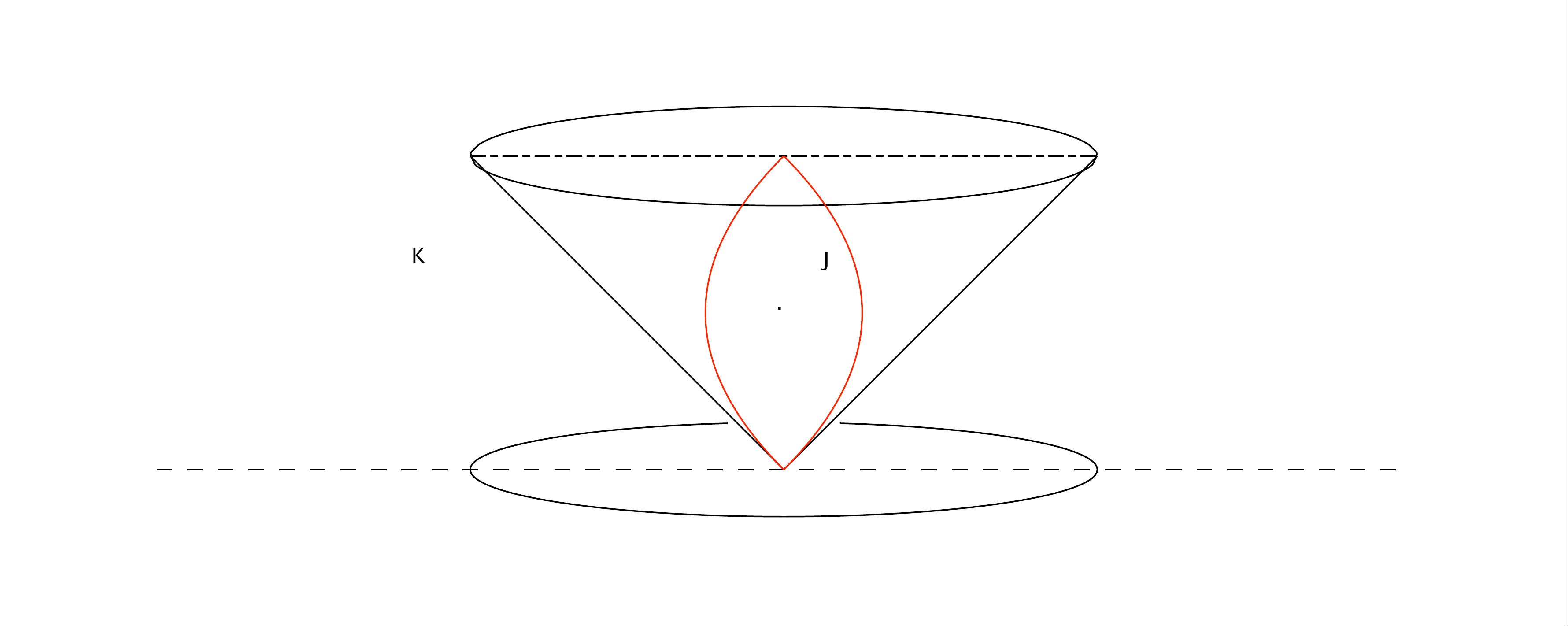}} 
\caption{ The spectrum of the cone algebra and $Z(\widehat g)=J$}
\end{figure}

So let 
$$\mbox{$f(z)=z+ i(|z|- \frac{1}{2})$ and $g(z)= z^2-|z|^2(1-|z|)^2$}.$$
 Then $f$ and $g$ 
belong to $A_{co}$. Their Gelfand transforms are given by
$$\mbox{$\widehat f(z,t)=z+ i (t-\frac{1}{2})$ and $\widehat g(z,t)=z^2-t^2(1-t)^2$}.$$
Since it is more convenient to work with $\R^2$-valued functions (instead of $\C$-valued ones)
when they are defined on $K$ ($K$ viewed as a subset of $\R^3$ instead of $\C\times \R)$,  we 
put 
$$F(x,y,t):=\Big({\rm Re}\,\widehat f(x+iy,t),{\rm Im}\,\widehat f(x+iy,t)\Big) $$
 and deduce 
 the following representation of the zero set of $\widehat g$ and the associated    action of $\widehat f$: 
$$Z(\widehat g)=\{\big(\pm t(1-t), t\big)\in \C\times \R: 0\leq t\leq 1\}=
\{\big(\pm t(1-t), 0, t\big)\in \R^3: 0\leq t\leq 1\},$$
and
\begin{eqnarray*}
F\big(\pm t(1-t), 0,  t\big)&=&
\Big({\rm Re}\,\widehat f \big(\pm t(1-t),  t\big),\; {\rm Im}\,\widehat f \big(\pm t(1-t),  t\big)\Big)\\
&=& \big(\pm t(1-t), t-\frac{1}{2}\big).
\end{eqnarray*} 
Then $J:=Z(\widehat g)$ is a Jordan curve contained in  $K$ and $J$ does not meet $Z(\widehat f)$.  Hence $(f,g)\in U_2(A_{co})$.  Moreover, we see that $F|_J$ is a translation map; in fact
 using complex coordinates  
in the plane $\{(x,y,t)\in \R^3: y=0\}$,  and putting  $w=x+it$, then
the action of $F$ on $J$ can be written as
$\tilde F(w)=w-i/2$, because with $x=\pm t(1-t)$, $\tilde F(x+it)=x+i(t-1/2)=w-i/2$.

In view of achieving a contradiction, 
suppose now that $(f,g)$ is reducible in $A_{co}$.  Then 
$$\widehat u:=\widehat f+ \widehat a\widehat g$$
 is a zero-free function  on $K$ for some $a,u\in A_{co}$. 
 Restricting $\widehat f$ to $Z(\widehat g)$, we find that the
translated   identity mapping on the Jordan curve $Z(\widehat g)$
has  a zero-free extension to the interior of that curve in the plane $y=0$.
Since $(0,0,1/2)$  is surrounded by that curve, we get 
a contradiction to Theorem  (C).
We conclude that the pair $(f,g)$ is not reducible in $A_{co}$ and so $\bsr A_{co}\geq 2$.
Putting all together,  $\bsr A_{co}=2$.
\end{proof}



 \section{The cylinder algebra}
 
Suppose that $\{(f_t,g_t): t\in [0,1]\}$ is a family of functions in $A(\bf D)$
such that $$Z(f_t)\inter Z(g_t)=\emp$$ for every $t$. By the Nullstellensatz for the disk algebra,
for each parameter $t$, there is  a solution $(x_t,y_t)\in A({\bf D})^2$ to the 
B\'ezout equation $x_tf_t+y_tg_t=1$. If the family $\{(f_t,g_t):t\in [0,1]\}$ 
depends continuously on $t$, 
do there exist solutions  to the B\'ezout equation that also depend continuously on $t$?
This problem has an affirmative answer and is best described by introducing the {\it cylinder algebra}:
\index{cylinder algebra}
$$\mbox{${\rm Cyl}({\D})=\big\{f\in C\big({\bf D}\times [0,1],\C\big): \;f(\cdot, t)\in A(\bf D)\;$ for all
$t\in [0,1] \big\}$}.$$

\begin{proposition}\label{cyl}
Let ${\rm Cyl}({\D})$ be the cylinder algebra. Then 
\begin{enumerate}
\item [(1)] ${\rm Cyl}(\D)$ is a uniformly closed subalgebra  of $C({\bf D}\times [0,1],\C)$.
\item [(2)] The set $\C[z,t]$ of  polynomials  of the form 
$$\sum_{j,k=0}^N a_{j,k} z^j t^k,~~  a_{j,k}\in \C,\;  N\in \N$$
is dense in ${\rm Cyl}(\D)$.
\item [(3)] $M({\rm Cyl}({\D}))=\{\delta_{(a,t)}: (a,t)\in {\bf D}\times [0,1]\}$, where
$$\delta_{(a,t)}: \begin{cases} {\rm Cyl}({\D}) & \to \C\\ f& \mapsto f(a,t).
\end{cases}$$
\item [(4)] An ideal $M$ in ${\rm Cyl}({\D})$ is maximal if and only if
$$M=M(z_0,t_0):=\{f\in {\rm Cyl}({\D}): f(z_0,t_0)=0\}$$
for some $(z_0,t_0)\in {\bf D}\times [0,1]$. In particular, ${\rm Cyl}({\D})$ is  natural on
${\bf D}\times [0,1]$.
\item [(5)]  Let $f_j\in {\rm Cyl}({\D})$, $j=1,\dots, n$.  Then the B\'ezout equation 
$\sum_{j=1}^n x_jf_j=1$  admits a solution in ${\rm Cyl}({\D})$ if and only if the functions $f_j$
do not have a common zero on  the cylinder ${\bf D}\times [0,1]$.
\end{enumerate}
\end{proposition}
\begin{proof}
(1)  This is clear.

(2) Let $f\in {\rm Cyl}({\bf D})$. Then,  for every fixed $t\in [0,1]$,  $f(\cdot ,t)\in A(\bf D)$
and so $f(\cdot,t)$  admits a Taylor series $\sum_{n=0}^\infty a_n(t) z^n$, where the 
Taylor coefficients are given by
\begin{eqnarray*}a_n(t)&=&\frac{1}{n!}\frac{\partial ^nf}{\partial^n z} (0,t)=
\frac{1}{2\pi i}\int_{|\xi|=1} \frac{f(\xi, t)}{\xi^{n+1}}d\xi.
\end{eqnarray*}
The uniform continuity of $f$ on ${\bf D}\times [0,1]$ now implies that
 $t\mapsto a_n(t)$ is a continuous function on $[0,1]$, because
\begin{eqnarray*}|a_n(t)-a_n(s)| &\leq& \frac{1}{2\pi }\int_{|\xi|=1}\frac{|f(\xi,t)-f(\xi,s)|}{|\xi|^{n+1}}\; |d\xi|\\
&\leq& \e \;\;\text{for $|s-t|<\delta$}.
\end{eqnarray*}
In particular $|a_n(t)|\leq ||f||_\infty$ for all $t\in [0,1]$ and $n\in\N$.
Weierstrass' theorem  now yields polynomials $p_n\in \C[t]$ such that
$$\mbox{$|p_n(t)-a_n(t)|< \e \;2^{-n} $ for every $t\in [0,1]$.}$$
We claim that  for $\rho\in ]0,1[$ sufficiently close to $1$  and $N$  sufficiently large,
the polynomial $q$ given by
$$q(z,t)=\sum_{n=0}^N p_n(t) \rho^n z^n$$
is uniformly close to $f(z,t)$. In fact,  due to uniform continuity again, 
we may  choose  $\rho\in ]0,1[$ so that $|f(z,t)-f(\rho z,t)|<\e$
for every $(z,t)\in {\bf D} \times [0,1]$. Hence 
\begin{eqnarray*}
|f(z,t)-q(z,t)| &\leq& |f(z,t)-f(\rho z, t)| + |f(\rho z,t) -q(z,t)|\\
&\leq & \e+ \sum_{n=0}^N |p_n(t)-a_n(t)| \;\rho^n\, |z|^n + \sum_{n=N+1}^\infty
|a_n(t)|\; \rho^n\, |z|^n\\
&\leq &\e+\e \sum_{n=0}^N 2^{-n} +||f||_\infty  \sum_{n=N+1}^\infty \rho^n\\
&\leq & 3\e  + ||f||_\infty  \frac{\rho^{N+1}}{1-\rho}\leq 4 \e
\end{eqnarray*}
for $N$ large.

(3)  Let $m\in M( {\rm Cyl}({\D}))$ and denote by $\bs c$  the coordinate function $\bs c(z,t):=z$
and by $\bs r$ the coordinate function $\bs r(z,t)=t$.  Note  that $\bs c, \bs r\in {\rm Cyl}({\D})$.
Let $$(z_0,t_0):=(m(\bs c), m(\bs r)).$$
Then $z_0\in  {\bf D}$ because $|m(\bs c)|\leq ||\bs c||_\infty=1$. Now $t_0\in \sigma(\bs r)$,
the spectrum of $\bs r$ in ${\rm Cyl}({\D})$.
 Because for $\lambda\in\C$
the function $\bs r-\lambda\in {\rm Cyl}({\D})^{-1}$ if and only if
$\lambda\notin [0,1]$, we see that $t_0=m(\bs r)\in [0,1]$.
Consequently, $(z_0,t_0)\in {\bf D} \times [0,1]$.

Given $f\in  {\rm Cyl}({\D})$,  let $(p_n)$ be  a sequence of polynomials in $ \C[z,t]$ converging uniformly on ${\bf D} \times [0,1]$ to $f$. Then
$$m(p_n)=p_n(z_0,t_0)\to f(z_0,t_0).$$
Hence $m=\delta_{(z_0,t_0)}$.

(4) and (5)  These assertions follow from Gelfand's theory.
\end{proof}
\bigskip

Recall that  the cylinder algebra was defined as
$${\rm Cyl}({\D})=\{f\in C({\bf D}\times [0,1],\C): \;f(\cdot, t)\in A(\bf D)\}.$$
For technical reasons, we  let vary $t$ now in the interval $[-1,1]$. 
In this subsection we determine the Bass and topological stable ranks of ${\rm Cyl}({\D})$
\footnote{ Corach and Su\'arez  determined  in \cite[p. 5]{cs} the Bass stable rank of $C([0,1], A(\D))$,
which coincides with ${\rm Cyl}(\D)$, by using advanced  methods from algebraic topology as well
as the Arens-Royden theorem.}.
The original question that led us to consider this algebra, was the following: 
let 
$$\F:= \{(f_t,g_t): t\in [-1,1]\}$$ 
be a family of disk-algebra functions with $Z(f_t)\inter Z(g_t)=\emp$.
Then by the Jones-Marshall-Wolff Theorem \cite{jmw}, 
for each parameter $t$, there is $(u_t,y_t)\in A({\bf D})^2$,
$u_t$ invertible,  such that
 $u_tf_t+y_tg_t=1$. If the family $\F$ depends continuously on $t$, 
do there exist solutions  to this type of the B\'ezout equation that also depend continuously on $t$?
Quite surprisingly, this is no longer the case. This stays in contrast to the unrestricted
B\'ezout equation $x_tf_t+y_tg_t=1$ dealt with in Proposition \ref{cyl}.  
Here is the outcome:

\begin{theorem}\label{bsr-cyl}
If ${\rm Cyl}(\D)$ is the cylinder algebra, then  $\bsr {\rm Cyl}(\D)=\tsr {\rm Cyl}(\D)=2$.
\end{theorem}
\begin{proof}
We first show that $\bsr  {\rm Cyl}(\D)\geq 2$.
Let $f(z,t)=z+it$ and $g(z,t)=z^2-(1-t^2)$. Then $(f,g)\in U_2({\rm Cyl}(\D))$, because
$$(z+it)(z-it)-g(z,t)=1.$$
 \begin{figure}[h]
   \hspace{3cm}
   \scalebox{.60} {\includegraphics{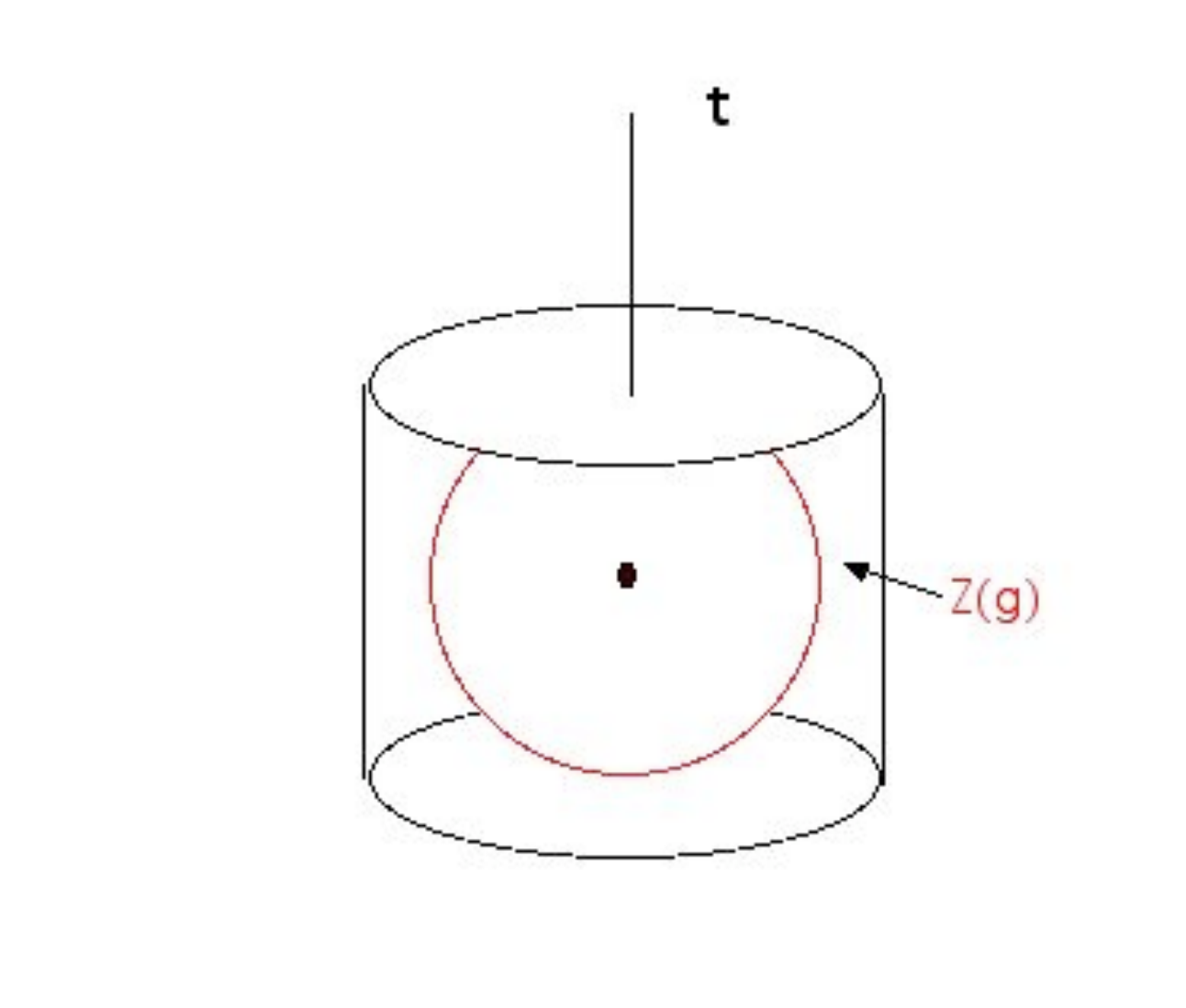}} 
\caption{The spectrum of the cylinder algebra}
\end{figure}
Suppose that $(f,g)$ is reducible. Then $u:=f+ ag$ is a zero-free function  on ${\bf D}\times [-1,1]$
for some $a\in {\rm Cyl}(\D)$. 
Now the zero set 
$$Z(g)=\big\{\big(\pm \sqrt{1-t^2},t\big): -1\leq t\leq 1\big\}=\{(x,y,t)\in \R^3: y=0, x^2+t^2=1\}$$
is a (vertical) circle. Restricting $f$ to $Z(g)$, and using complex 
coordinates $w$ on the disk $D$ formed by $Z(g)$, we obtain with 
$w=\pm  \sqrt{1-t^2}+it$ , that 
$$F(w):=f(\pm \sqrt{1-t^2}, t)= \pm\sqrt{1-t^2}+it=w.$$
Thus $f$ is the identity mapping on the circle $Z(g)$ and $u|_D$ is a zero-free extension
of $f|_{Z(g)}$. This is a contradiction to Theorem (C).\\

Next we prove that $\tsr {\rm Cyl}(\D)\leq 2$. Let $(f,g)\in {\rm Cyl}(\D)^2$. According to
Proposition \ref{cyl}, let $\bs F:=(p,q)\in (\C[z,t])^2$ be chosen so that 
$$||p-f||_\infty+||q-g||_\infty<\e.$$
By Theorem (A), $\bs F(\R^3)\ss\C^2$ has $4$-dimensional Lebesgue measure zero. 
Hence there is a null-sequence $(\e_n,\eta_n)$ in $\C^2$  such that 
$$(\e_n,\eta_n)\notin \bs F(\R^3).$$
Consequently, the pairs
$$(p-\e_n, q-\eta_n)$$
are invertible in ${\rm Cyl}({\D})$ by Proposition \ref{cyl} (5). Thus $\tsr {\rm Cyl}({\D})\leq 2$.

Combining both facts, we deduce that $\bsr {\rm Cyl}({\D})=\tsr {\rm Cyl}({\D})=2$.
\end{proof}

 \end{document}